\documentclass{amsart}
\makeatletter
\def\LaTeX{\leavevmode L\raise.42ex
   \hbox{\kern-.3em\size{\sf@size}{0pt}\selectfont A}\kern-.15em\TeX}
\makeatother

\newcommand{\BibTeX}{{\rm B\kern-.05em{\sc
i\kern-.025emb}\kern-.08em\TeX}}
\newtheorem{col}{Corollary}[section]
\newtheorem{thm}{Theorem}[section]
\newtheorem{lem}[thm]{Lemma}
\newtheorem{dfn}[thm]{Definition}
\theoremstyle{definition}

\numberwithin{equation}{section}

\begin{document}

\title[
Bernstein-Nikolskii and Riesz 
on compact homogeneous manifolds]{
Bernstein-Nikolskii inequalities  and Riesz interpolation formula
on compact homogeneous manifolds}

\author{Isaac Pesenson}
\address{Department of Mathematics, Temple University,
Philadelphia, PA 19122} \email{pesenson@math.temple.edu}


\begin{abstract}
Bernstein-Nikolskii inequalities and Riesz interpolation formula
are established for eigenfunctions of Laplace operators and
polynomials on compact homogeneous manifolds.
\end{abstract}

\maketitle

\section{Introduction}

 Consider a trigonometric
polynomial $ \textit{T}$ of one variable $t$ as a function on a
unit circle $\mathbb{S}$. For its derivatives of
$\textit{T}^{(k)}$ the so-called Bernstein and Bernstein-Nikolskii
inequalities hold true
\begin{equation}
\|\textit{T}^{(k)}\|_{L_{p}(\mathbb{S})}\leq n^{k}
\|\textit{T}\|_{L_{p}(\mathbb{S})}
\end{equation}
and
\begin{equation}
\|\textit{T}^{(k)}\|_{L_{q}(\mathbb{S})}\leq 3n^{k+1/p-1/q}
\|\textit{T}\|_{L_{p}(\mathbb{S})},
\end{equation}
where $n$ is the order of $\textit{T}$ and $1\leq p\leq q\leq
\infty$. The constant 3 is not the best but the inequality is
exact in the sense that for the Feyer kernel
$$
 F_{n}(t)=\frac{1}{n+1}\frac{\sin^{2}
 \frac{n+1}{2}t}{2\sin^{2}\frac{t}{2}}, t\in \mathbb{S},
$$
one has
$$
\|F_{n}^{(k)}\|_{L_{q}(\mathbb{S})}=C_{p,q}n^{k+1/p-1/q}
\|F_{n}\|_{L_{p}(\mathbb{S})}.
$$
 The Bernstein inequality (1.1) can be obtained as a
consequence of the Riesz interpolation formula

\begin{equation}
\frac{d\textit{T}(t)}{dt}=\frac{1}{4\pi}\sum_{k=1}^{2n}(-1)^{k+1}
\frac{1}{\sin^{2}\frac{t_{k}}{2}}\textit{T}(t+t_{k}), t\in
\mathbb{S}, t_{k}=\frac{2k-1}{2n}\pi.
\end{equation}
If one will treat $\textit{T}$ as an entire function of
exponential type on $\mathbb{C}$ which is bounded on the real
axis, then the Riesz interpolation formula can be written in the
form
\begin{equation}
\frac{d\textit{T}(t)}{dt}=\frac{n}{\pi^{2}}\sum_{k\in\mathbb{Z}}\frac{(-1)^{k-1}}{(k-1/2)^{2}}
\textit{T}(t+\frac{\pi}{n}\left(k-1/2)\right), t\in \mathbb{R}.
\end{equation}

The Bernstein-Nikolskii inequality (1.2) is a consequence of the
inequality (1.1) and the following inequality which is known as
the Nikolskii inequality
\begin{equation}
\|\textit{T}\|_{L_{p}(\mathbb{S})}\leq \max_{u \in
\mathbb{S}}\left(h\sum_{k=1}^{N}
\left|\textit{T}(kh-u)\right|^{p}\right)^{1/p}\leq
(1+nh)\|\textit{T}\|_{L_{p}(\mathbb{S})},
\end{equation}
where $h=2\pi/N,  N\in \mathbb{N}, 1\leq p\leq \infty.$ Similar
results hold true for the $m$-dimensional torus
$\mathbb{T}^{m}=\mathbb{S}\times ...\times \mathbb{S}$. The
inequalities (1.1)- (1.5) and their proofs can be found in
\cite{A}, Ch. 4, and in  \cite{N}, Ch. 2 and 3.

Trigonometric polynomials can be characterized as eigenfunctions
of the Laplace operator on  $\mathbb{T}^{m}$. On the other hand,
if one considers the equivariant embedding of $\mathbb{T}^{m}$
into  Euclidean space $\mathbb{R}^{2m}$ (flat torus) then every
trigonometric polynomial on $\mathbb{T}^{m}$ can be identified
with a restriction to $\mathbb{T}^{m}$ of an algebraic polynomial
in the ambient space in $\mathbb{R}^{2m}$.

All  results listed above are at the very core of the classical
approximation theory. The goal of the present article is to obtain
similar results for a compact homogeneous manifold $M$.

Very deep  generalizations of some ideas which intimately relate
to  Bernstein-Markov type inequalities were obtained by J.
Bourgain \cite{B}, A. Brudnyi \cite{B1}, \cite{B2}, A. Carbery and
J. Wright \cite{CW}. In particular, the results of A. Brudnyi can
be used to obtain  a version of our Theorem 3.2. An abstract
approach to Bernstein inequality was suggested by A. Gorin
\cite{G}.

The Bernstein
 inequality
 on compact homogeneous manifolds was developed and explored in
\cite{BKLT}, \cite{BLMT},   \cite{Kam}, \cite{Pes0}- \cite{Pes3},
\cite{R}. In particular, the Bernstein inequality on spheres was
considered in \cite{D1}-\cite{D3}, \cite{MNW}. The
Bernstein-Nikolskii-type inequalities on compact symmetric spaces
of rank one were considered in  interesting papers \cite{BDai},
\cite{Dai}.
 But as well as we know  nobody considered generalizations of (1.5). In fact the
inequality we prove (see (1.13) bellow) even more general than
(1.5) and seems to be new even in the case of trigonometric
polynomials on a torus. Our approach to the Bernstein inequality
and the Riesz interpolation formula is closer to the classical one
in the sense that we are using  first-order differential operators
instead of using the Laplace-Beltrami operator as it was done in
\cite{Kam}. Note that generalizations of Bernstein-Nikolskii
inequalities to non-compact symmetric spaces will appear in our
paper \cite{Pes4}.

In what follows we introduce some very basic notions of harmonic
analysis on compact homogeneous manifolds \cite{H3}, Ch. II. More
details on this subject can be found, for example, in
 \cite{V}, \cite{Z}.

 Let $M, dim M=m,$ be a
compact connected $C^{\infty}$-manifold. It says  that a compact
Lie group $G$ effectively acts on $M$ as a group of
diffeomorphisms if

1)  every element $g\in G$ can be identified with a diffeomorphism
$$
g: M\rightarrow M
$$
of $M$ onto itself and
$$
g_{1}g_{2}\cdot x=g_{1}\cdot(g_{2}\cdot x), g_{1}, g_{2}\in G,
x\in M,
$$
where $g_{1}g_{2}$ is the product in $G$ and $g\cdot x$ is the
image of $x$ under $g$,

2) the identity $e\in G$ corresponds to the trivial diffeomorphism
\begin{equation}
e\cdot x=x,
\end{equation}

3) for every $g\in G, g\neq e,$ there exists a point $x\in M$ such
that $g\cdot x\neq x$.

\bigskip

A group $G$ acts on $M$ \textit{transitively} if in addition to
1)- 3) the following property holds

4) for any two points $x,y\in M$ there exists a diffeomorphism
$g\in G$ such that
$$
g\cdot x=y.
$$

A \textit{homogeneous} compact manifold $M$ is an
$C^{\infty}$-compact manifold on which transitively acts a compact
Lie group $G$. In this case $M$ is necessary of the form $G/K$,
where $K$ is a closed subgroup of $G$. The notation $L_{p}(M),
1\leq p\leq \infty,$ is used for the usual Banach  spaces
$L_{p}(M,dx), 1\leq p\leq \infty$, where $dx$ is an invariant
measure.

Every element $X$ of the Lie algebra of $G$ generates a vector
field on $M$ which we will denote by the same letter $X$. Namely,
for a smooth function $f$ on $M$ one has
\begin{equation}\label{vf}
 Xf(x)=\lim_{t\rightarrow 0}\frac{f(\exp tX \cdot x)-f(x)}{t}
 \end{equation}
for every $x\in M$. In the future we will consider on $M$ only
such vector fields. Translations along integral curves of such
vector field $X$ on $M$  can be identified with a one-parameter
group of diffeomorphisms of $M$ which is usually denoted as $\exp
tX, -\infty<t<\infty$. At the same time the one-parameter group
$\exp tX, -\infty<t<\infty,$ can be treated as a strongly
continuous one-parameter group of operators in a space $L_{p}(M),
1\leq p\leq \infty$ which acts on functions according to the
formula
$$
f\rightarrow f(\exp tX\cdot x), t\in \mathbb{R}, f\in L_{p}(M),
x\in M.
$$
 The
generator of this one-parameter group will be denoted as $D_{X,p}$
and the group itself will be denoted as
$$
e^{tD_{X,p}}f(x)=f(\exp tX\cdot x), t\in \mathbb{R}, f\in
L_{p}(M), x\in M.
$$

According to the general theory of one-parameter groups in Banach
spaces \cite{BB}, Ch. I,  the operator $D_{X,p}$ is a closed
operator in every $L_{p}(M), 1\leq p\leq \infty.$ In order to
simplify notations we will often use notation $D_{X}$ instead of
$D_{X, p}$.

It is known (\cite{H2},
 Ch. V, proof of the Theorem 3.1,) that
 on every compact homogeneous manifold $M=G/K$ there exist vector fields
  $X_{1},X_{2},..., X_{d}, d=dim G,$ such that the second order  differential
   operator on $M$
   $$
X_{1}^{2}+ X_{2}^{2}+ ...+ X_{d}^{2}, d=dim G,
   $$
   commutes with all $X_{1},...,X_{d}$. The
   corresponding operator in $L_{p}(M), 1\leq p\leq\infty,$
\begin{equation}
-\mathcal{L}=D_{1}^{2}+ D_{2}^{2}+ ...+ D_{d}^{2},
D_{j}=D_{X_{j}}, d=dim G,\label{Laplacian}
\end{equation}
commutes with all operators $D_{j}=D_{X_{j}}$. This operator
$\mathcal{L}$ which is usually called the Laplace operator is
involved in most of constructions and results of our paper.

In some situations the operator $\mathcal{L}$ is essentially the
Laplace-Beltrami operator of an invariant metric on $M$. It
happens for example in the following cases.

1) If $M$ is a $d$-dimensional torus and $-\mathcal{L}$ is the sum
of squares of partial derivatives.

2) If the manifold $M$ is itself a group $G$ which is compact and
semi-simple then $-\mathcal{L}$ is exactly the Laplace-Beltrami
operator of an invariant metric on $G$ (\cite{H2}, Ch. II,
Exercise A4).

3) If $M=G/K$ is a compact symmetric space of
  rank one then the operator
 $-\mathcal{L}$ is proportional to the Laplace-Beltrami operator
of an invariant metric on $G/K$. It follows from the fact that in
the rank one case every second-order operator  which commutes with
all invariant vector fields is proportional to the
Laplace-Beltrami operator (\cite{H2}, Ch. II, Theorem 4.11).

Let us stress one more time that in the present paper we  use only
the property that the operator $\mathcal{L}$ commutes with all
 vector fields $X_{1},...,X_{d}$ on $M$  and we do not explore its
relations to the Laplace-Beltrami operator of the invariant
metric.

Note that if $M=G/K$ is a compact symmetric space then the number
$d=dim G$ of operators in the formula (\ref{Laplacian}) can be
strictly bigger than the dimension $ m=dim M$. For example on a
two-dimensional sphere $\mathbb{S}^{2}$ the Laplace-Beltrami
operator $L_{\mathbb{S}^{2}}$ can be written as
\begin{equation}
\mathcal{L}_{\mathbb{S}^{2}}=D_{1}^{2}+ D_{2}^{2}+
D_{3}^{2},\label{S-Laplacian}
\end{equation}
where $D_{i}, i=1,2,3,$ generates a rotation  in $\mathbb{R}^{3}$
around coordinate axis $x_{i}$:
\begin{equation}
D_{i}=x_{j}\partial_{k}-x_{k}\partial_{j},
\end{equation}
where $j,k\neq i.$

 The  operator $\mathcal{L}$  is an elliptic differential operator which
is defined on $C^{\infty}(M)$ and we will use the same notation
$\mathcal{L}$ for its closure from $C^{\infty}(M)$ in $L_{p}(M),
1\leq p\leq \infty$. In the case $p=2$ this closure is a
self-adjoint positive definite operator in the space $L_{2}(M)$.
The spectrum of this operator is discrete and goes to infinity
$0=\lambda_{0}<\lambda_{1}\leq \lambda_{2}\leq ...$, where
 we count each eigenvalue with its  multiplicity.  For eigenvectors corresponding 
 to eigenvalue $\lambda_{j}$ we will use notation $\varphi_{j}$, i. e. 
 \begin{equation}
 \mathcal{L}\varphi_{j}=\lambda_{j}\varphi_{j}.
 \end{equation}
Let $\varphi_{0}, \varphi_{1}, \varphi_{2}, ...$ be a corresponding
complete system of orthonormal eigenfunctions and
$\textbf{E}_{\omega}(\mathcal{L}), \omega>0,$ be a span of all
eigenfunctions of $\mathcal{L}$ whose corresponding eigenvalues
are not greater $\omega$.

In the rest of the paper the notations $\mathbb{D}=\{D_{1},...,
D_{d}\}, d=dim G,$ will be used for differential operators in
$L_{p}(M), 1\leq p\leq \infty,$ which are involved in the formula
(\ref{Laplacian}).

\begin{dfn}
We say that a  function $f\in L_{p}(M), 1\leq p\leq \infty,$
belongs to the Bernstein space
$\mathbf{B}_{\omega}^{p}(\mathbb{D}),
\mathbb{D}=\{D_{1},...,D_{d}\}, d=dim G,$ if and only if for every
$1\leq i_{1},...i_{k}\leq d$ the following Bernstein inequality
holds true
 \begin{equation}
 \|D_{i_{1}}...D_{i_{k}}f\|_{p}\leq
 \omega^{k}\|f\|_{p}, k\in \mathbb{N}.\label{Bern}
\end{equation}
\end{dfn}
\begin{dfn}
We say that a  function $f\in L_{p}(M), 1\leq p\leq \infty,$
belongs to the Bernstein space
$\mathbf{B}_{\omega}^{p}(\mathcal{L}), $ if and only if for every
$k\in \mathbb{N}$ the following Bernstein inequality holds true
$$
\|\mathcal{L}^{k}f\|_{p}\leq \omega^{k}\|f\|_{p}, k\in \mathbb{N}.
$$
\end{dfn}

Since $\mathcal{L}$ in the space $L_{2}(M)$ is self-adjoint and
positive-definite there exists a unique positive square root
$\mathcal{L}^{1/2}$. In this case the last inequality is
equivalent to the inequality
$$
\|\mathcal{L}^{k/2}f\|_{2}\leq \omega^{k/2}\|f\|_{2}, k\in
\mathbb{N}.
$$

Note that at this point it is not clear if the Bernstein spaces
$\mathbf{B}_{\omega}^{p}(\mathbb{D}),
\mathbf{B}_{\omega}^{p}(\mathcal{L})$ are linear spaces. These
facts will be established later in the Lemma 2.1 and Theorem 2.2.

 The following Lemma was proved in \cite{Pes1} for
 any  homogeneous manifold.
\begin{lem} There exists a constant $N(M)$ such that
for any   sufficiently small $r>0$ there
 exists a set of points $\{x_{i}\}$ from $M$ such that

1) balls $B(x_{i}, r)$ are disjoint,

2) balls $B(x_{i}, 2r)$ form a cover of $M$,

3) multiplicity of the cover by balls $B(x_{i}, 4r)$ is not
greater $N(M).$

\end{lem}

\begin{dfn}
We will use notation $Z( r, N(M))$ for a set of points
$\{x_{i}\}\in M$ which satisfies the properties 1)- 3) from the
last Lemma 1.1 and we will call such set a  $(r,N(M))$-lattice of
$M$.
\end{dfn}

\begin{dfn}
We will use notation $Z_{G}( r, N(M))$ for a set of elements
$\{g_{\nu}\}$ of the group $G$  such that the points
$\{x_{\nu}=g_{\nu}\cdot o\}$ form a $(r,N(M))$-lattice in $M$
(here $\{o\}\in M$ is the origin of $M$). Such set $Z_{G}( r,
N(M))$ will be  called a $(r,N(M))$-lattice in $G$.
\end{dfn}

Our main results are the following. In the section 2 we establish
Riesz interpolation formula for Bernstein spaces
$\mathbf{B}_{\omega}^{p}(\mathbb{D}),
\mathbb{D}=\{D_{1},...,D_{d}\}, d=dim G,$ and use this formula to
prove some basic properties of the Bernstein spaces.

\begin{thm} The following conditions are equivalent for any $1\leq p\leq \infty$:

1) $f\in \mathbf{B}_{\omega}^{p}(\mathbb{D}),
\mathbb{D}=\{D_{1},...,D_{d}\}, d=dim G,$

2) for any  $1\leq i_{1},... ,i_{k}\leq d$,  any $1\leq j\leq d,$
and any functional $\psi^{*}\in L_{p}(M)^{*}, 1\leq p\leq \infty,$
the function
$$
\left<e^{tD_{j}}D_{i_{1}}...D_{i_{k}}f, \psi^{*}\right>:
\mathbb{R}\rightarrow \mathbb{R},
$$
  of the real variable $t$ has an extension to the
complex plane $\mathbb{C}$ as an entire function of the
exponential type at most $\omega$ and is bounded on the real line,

3) the following Riesz interpolation formula holds true
\begin{equation}
D_{i_{1}}...D_{i_{k}}f=\mathcal{R}_{i_{1}}^{\omega}...\mathcal{R}_{i_{k}}^{\omega}f,
1\leq i_{k}\leq d,\label{Rieszk}
\end{equation}
were
\begin{equation}
\mathcal{R}_{i}^{\omega}f=\frac{\omega}{\pi^{2}}\sum_{j\in\mathbb{Z}}\frac{(-1)^{j-1}}{(j-1/2)^{2}}
e^{\left(\frac{\pi}{\omega}(j-1/2)\right)D_{i}}f,\label{R1}
\end{equation}
were $e^{tD_{X}}f(x)=f(\exp tX\cdot x), t\in \mathbb{R}, f\in
L_{p}(M), x\in M,$ and convergence in (\ref{R1}) is understood in
the $L_{p}(M)$-sense.

\end{thm}

 It is also shown  that if $M$ is equivariantly
embedded into Euclidean space $\mathbb{R}^{N}$ and
$\textbf{P}_{n}(M)$ is the set of restrictions to $M$ of
polynomials in $\mathbb{R}^{N}$ of order $n$ then for any $f\in
\textbf{P}_{n}(M)$ the following  Riesz interpolation formula
 holds true
$$
D_{i_{1}}...D_{i_{k}}f=\mathcal{R}_{i_{1}}^{n}...\mathcal{R}_{i_{k}}^{n}f,
1\leq i_{k}\leq d.
$$

In particular
$$
\mathcal{L}f=\sum_{l,
k=1}^{d}\mathcal{R}_{i_{l}}^{n}\left(\mathcal{R}_{i_{k}}^{n}f\right),
f\in \textbf{P}_{n}(M) .
$$
For example, in the case of the unit two-dimensional sphere
$\mathbb{S}^{2}$ with the standard embedding into $\mathbb{R}^{3}$
the last formula means that the function $\mathcal{L}f$ where $f$
is a polynomial can be calculated by using a combination of
translations of $f$ with respect to rotations around coordinate
axes. Using the Riesz interpolation formula  we also prove
Bernstein inequality for polynomials on compact homogeneous
manifolds equivariantly embedded into Euclidean space.

In the section 3 we prove a Nikolskii-type inequality. Namely, we
show that \textit{ for any $1\leq p\leq \infty,$ any natural
$l>m/p, $ there exists a constant $ C(M,l)$ such that for
 any $\omega>0$ any
$(r,N(G))$-lattice $Z_{G}( r, N(G))\subset
 G$ with sufficiently small $r>0$, and any $q\geq p$ the following
 inequalities hold true
\begin{equation}
 \|f\|_{q}\leq r^{m/q} \sup_{g\in
G}\left(\sum_{g_{i}\in Z_{G}( r, N(G))}\left(|f(g_{i}g\cdot
o)|\right)^{p}\right)^{1/p}\leq
$$
$$
C(M,  l)r^{m/q-m/p}\left(1+(r\omega)^{l}\right)\|f\|_{p}, m=dim
M,\label{Nik1}
\end{equation}
for all $f\in  \mathbf{B}_{\omega}^{p}(\mathbb{D}) $. In
particular, these inequalities hold true for polynomials in $
\textbf{P}_{n}(M)$ with $\omega=n$.}

Using these Nikolskii-type inequalities we prove that for any
$1\leq p, q\leq \infty$ the following equality holds true
$$
\mathbf{B}_{\omega}^{p}(\mathbb{D})
=\mathbf{B}_{\omega}^{q}(\mathbb{D})\equiv
\mathbf{B}_{\omega}(\mathbb{D}),\mathbb{D}=\{D_{1},...,D_{d}\},
d=dim G,
$$
which means that if the Bernstein-type inequalities (\ref{Bern})
are satisfied for a single  $1\leq p\leq \infty$, then they are
satisfied for all $1\leq p\leq \infty$.

The inequalities (\ref{Bern}) and (\ref{Nik1}) are used to obtain
the following  inequality  of the Bernstein-Nikolskii-type
$$
\|D_{i_{1}}...D_{i_{k}}f\|_{q}\leq
C(M)\omega^{k+\frac{m}{p}-\frac{m}{q}}\|f\|_{p}, f\in
\mathbf{B}_{\omega}(\mathbb{D}), m=dim M,
$$
for a certain constant $C(M)$ and any $1\leq p\leq q\leq \infty,
1\leq i_{1},...i_{k}\leq d, k\in \mathbb{N}, d=dim G.$

 We also
prove the following embeddings which describe relations between
Bernstein spaces $
\textbf{B}_{n}(\mathbb{D}),\mathbb{D}=\{D_{1},...,D_{d}\}, d=dim
G,$ and eigen spaces $\textbf{E}_{\lambda}(\mathcal{L})$ for $
-\mathcal{L}=D_{1}^{2}+ D_{2}^{2}+ ...+ D_{d}^{2}, d=dim G,$
$$
\textbf{B}_{\omega}(\mathbb{D})\subset\textbf{E}_{\omega^{2}d}(\mathcal{L})\subset
\textbf{B}_{\omega\sqrt{d}}(\mathbb{D}), d=dim G,\omega>0.
$$
These embeddings obviously  imply the equality
$$
\bigcup_{\omega>0} \textbf{B}_{\omega}(\mathbb{D})=\bigcup_{j}
\textbf{E}_{\lambda_{j}}(\mathcal{L}),
$$
which means  \textit{that a function on $M$ satisfies a Bernstein
inequality (\ref{Bern}) in a norm of $L_{p}(M), 1\leq p\leq
\infty,$ if and only if it is a linear combination of
eigenfunctions of $\mathcal{L}$.}

As a consequence we obtain the following inequalities
\begin{equation}
\|\mathcal{L}^{k}\varphi\|_{p}\leq
(d\omega)^{2k}\|\varphi\|_{p}, k\in \mathbb{N,} d=dim G,
\end{equation}
 for every $\varphi \in \textbf{E}_{\omega}(\mathcal{L})$,$1\leq p\leq \infty$.

 Note that in the
 case of homogeneous manifolds of rank one a better constant for such inequality
 was given by A. Kamzolov \cite{Kam}.

 Another consequence of our Bernstein-Nikolskii inequality  is
 the following estimate for every $\varphi \in \textbf{E}_{\omega}(\mathcal{L})$
$$
\|\mathcal{L}^{k}\varphi\|_{q}\leq C(M)
\omega^{2k+\frac{m}{p}-\frac{m}{q}}\|\varphi\|_{p}, k\in
\mathbb{N}, m=dim M, d=dim G, 1\leq p\leq q\leq\infty,
$$
for a certain constant $C(M)$ which depends just on the manifold.
At the end of the paper we establish the following relations
$$
\textbf{P}_{n}(M)\subset \textbf{B}_{n}(\mathbb{D})\subset
\textbf{E}_{n^{2}d}(\mathcal{L})\subset
\textbf{B}_{n\sqrt{d}}(\mathbb{D}), d=dim G,n\in \mathbb{N},
$$
and
$$
\bigcup _{n}\textbf{P}_{n}(M)=\bigcup_{\omega}
\textbf{B}_{\omega}(\mathbb{D})=\bigcup_{j}
\textbf{E}_{\lambda_{j}}(\mathcal{L}), n\in \mathbb{N,}
$$
where $ \textbf{P}_{n}(M)$ is the space of polynomials. Note that
the embedding
\begin{equation}
\textbf{P}_{n}(M)\subset
\mathbf{B}_{\omega}^{\infty}(\mathbb{D})\label{P-B}
\end{equation}
was proved by D. Ragozin \cite{R}.

\section{Bernstein inequality and Riesz interpolation formula on
compact homogeneous manifolds}

We assume that $A$ is a generator of one-parameter group of
isometries $e^{ tA}$ in a Banach space $E$ with the norm $\|\cdot
\|$. The Bernstein space
 $\mathbf{B}_{\omega}(A), \omega>0,$ is introduced as  a set of all vectors $f$ in $E$ for
 which
\begin{equation}
\|A^{k}f\|\leq \omega^{k}\|f\|, k\in \mathbb{N}.
\end{equation}
Let's introduce the operator
$$
\mathcal{R}_{A}^{\omega}f=\frac{\omega}{\pi^{2}}\sum_{k\in\mathbb{Z}}\frac{(-1)^{k-1}}{(k-1/2)^{2}}
e^{\left(\frac{\pi}{\omega}(k-1/2)\right)A}f, f\in E, \omega>0.
\label{Riesz1}
$$
Since $\|e^{tA}f\|=\|f\|, f\in E,$ and since the following
identity holds
\begin{equation}
\frac{\omega}{\pi^{2}}\sum_{k\in\mathbb{Z}}\frac{1}{(k-1/2)^{2}}=\omega,\label{id}
 \end{equation}
 the operator $\mathcal{R}_{A}^{\omega}, \omega>0,$ is a bounded operator in $E$ and
 \begin{equation}
 \|\mathcal{R}_{A}^{\omega}f\|\leq \omega\|f\|, f\in
 E.\label{Riesznorm}
 \end{equation}
\begin{lem} The following conditions are equivalent:

1) $f\in \mathbf{B}_{\omega}(A);$

2) for any functional $\psi^{*}$ from the dual space $ E^{*}$ and
for any $n\in \mathbb{N}$ the function
\begin{equation}
F_{n}(t)=\left<e^{tA}A^{n}f,\psi^{*}\right>:\mathbb{R}\rightarrow\label{der}
\mathbb{R},
\end{equation}
has an extension to the complex plane
$\mathbb{C}$ as an entire function of the exponential type at most
$\omega$ and is bounded on the real line;

3) the following Riesz interpolation formula holds true
\begin{equation}
A^{n}f=\left(\mathcal{R}_{A}^{\omega}\right)^{n}f, n\in
\mathbb{N}. \label{Rieszn}
\end{equation}

\end{lem}

\begin{proof}

Let us assume that $f\in \mathbf{B}_{\omega}(A)$. According to a
general theory of one-parameter groups of operators in Banach
spaces \cite{BB}, Ch.1,
$$
\frac{d}{dt}e^{tA}f=Ae^{tA}f.
$$
Since
$$
\frac{d}{dt}F(t)=\frac{d}{dt}\left<e^{tA}f,\psi^{*}\right>=
\left<\frac{d}{dt}e^{tA}f,\psi^{*}\right>=\left<Ae^{tA}f,\psi^{*}\right>,
$$
it implies that if $f\in \mathbf{B}_{\omega}(A)$ then  for any
functional $\psi^{\ast}\in E^{\ast}$  the scalar function
$$
F(z)=\left< e^{zA}f, \psi^{\ast }\right>
$$
is entire because its Taylor series at $t=0$ is  the series
\begin{equation}
F(z)=\left<e^{zA}f, \psi^{\ast}\right>=\sum^{\infty}_{l=0}
\frac{z^{l}\left<A^{l}f, \psi^{\ast}\right>}{l!}
\end{equation}
which converges because of the estimate
$$
|\left<A^{l}f,\psi^{\ast}\right>|\leq
\|\psi^{\ast}\|\|A^{l}f\|\leq \omega^{l}\|\psi^{\ast}\|\|f\|.
$$
The last estimate also implies the  inequality
\begin{equation}
|F(z)|\leq e^{|z|\omega}\|\psi^{\ast}\|\|f\|,\psi^{\ast}\in
E^{\ast},
\end{equation}
which shows that $F(t)=\left<e^{tA} f,\psi^{*}\right>$ has an
extension to the complex plane $\mathbb{C}$ as an entire function
of the exponential type at most $\omega$.

Moreover, because the group $e^{tA}$ is an isometry group we have
for any real $t$ the inequality
$$
|F(t)|=|\left< e^{tA}f, \psi^{\ast }\right>|\leq\|\psi^{\ast}\|\|
e^{tA}f\|\leq\|\psi^{\ast}\|\|f\|,
$$
which shows that the function $F$ is bounded on the real line.
Thus, we proved that if $f\in \mathbf{B}_{\omega}(A)$ then the
function $F_{k}(t)=\left<e^{tA}A^{k}f,\psi^{*}\right>$ is an
entire function of the exponential type $\omega$ which is bounded
on the real line. The similar statement about a general function
of the form (\ref{der}) follows from the fact that the set
$\mathbf{B}_{\omega}(A)$ is obviously invariant under the operator
$A$. The implication $1)\rightarrow 2)$ is proved.

 If the property 2)
holds true then  for the function
$F_{n-1}(t)=\left<e^{tA}A^{n-1}f,\psi^{*}\right>, n\in
\mathbb{N},$ the following classical Riesz interpolation formula
takes place
$$
\frac{d}{dt}F_{n-1}(t)=\frac{\omega}{\pi^{2}}\sum_{k\in\mathbb{Z}}\frac{(-1)^{k-1}}{(k-1/2)^{2}}
F_{n-1}(t+\frac{\pi}{\omega}\left(k-1/2)\right).
$$
Using the same arguments as above we can write this formula in the
following form
$$
\left<Ae^{tA}A^{n-1}f,\psi^{*}\right>=\frac{\omega}{\pi^{2}}\sum_{k\in\mathbb{Z}}\frac{(-1)^{k-1}}{(k-1/2)^{2}}
\left<e^{(t+\frac{\pi}{\omega}(k-1/2))A}A^{n-1}f,\psi^{*}\right>.
$$
For $t=0$ it gives
\begin{equation}
\left<A^{n}f,\psi^{*}\right>=\frac{\omega}{\pi^{2}}\sum_{k\in\mathbb{Z}}\frac{(-1)^{k-1}}{(k-1/2)^{2}}
\left<e^{\left(\frac{\pi}{\omega}(k-1/2)\right)A}A^{n-1}f,\psi^{*}\right>.
\end{equation}
Since the last formula holds for any functional $\psi^{*}\in
E^{*}$ it proves the equality
\begin{equation}A^{n}f=\mathcal{R}_{A}^{\omega}A^{n-1}f, n\in
\mathbb{N},
\end{equation}
for every function for which the property 2) holds. But then
$$
A^{n}f=A(A(...Af))=\mathcal{R}_{A}^{\omega}(\mathcal{R}_{A}^{\omega}
(...\mathcal{R}_{A}^{\omega}f))= (\mathcal{R}_{A}^{\omega})^{n}f,
n\in \mathbb{N}.
$$

The implication $2)$ $\rightarrow $ $3)$ is proved.

To finish the proof of the Theorem we have to show that 3) implies
1). But this fact easily follows from  the formulas (\ref{Rieszn})
and (\ref{Riesznorm}).
\end{proof}

As a consequence of this Theorem we obtain the following
Corollary.
\begin{col}
If $A$ is a generator of one-parameter group of isometries $e^{
tA}$ in a Banach space $E$ then the Bernstein spaces
$\mathbf{B}_{\omega}(A)$ are linear and closed for every
$\omega>0$.
\end{col}

 We return to a
homogeneous manifold $M$ and we are going to use notations which
were developed in the Introduction.

\begin{thm}
The set $\mathbf{B}_{\omega}^{p}(\mathbb{D}), \mathbb{D}=\{D_{1},...,D_{d}\}, d=dim G,1\leq p\leq \infty,$
has the following properties:

1) it is invariant under every $D_{\nu}, 1\leq\nu\leq d$;

2) it is a linear  subspace of $L_{p}(M)$;

3) it is a closed subspace of $L_{p}(M)$.
\end{thm}

\begin{proof}
To prove the first part of the Theorem we have to show that if
$f\in \mathbf{B}_{\omega}^{p}(\mathbb{D}),
 1\leq p\leq \infty,$ and $g=D_{\nu}f,$ for a $1\leq \nu\leq d,$
then the following inequality holds true
\begin{equation}
\|D_{i_{k}}...D_{i_{1}}g\|_{p}\leq \omega^{k}\|g\|_{p},
g=D_{\nu}f,
\end{equation}
for any $1\leq i_{1},i_{2}...,i_{k}\leq d $ . First we are going
to show that if $f\in \mathbf{B}_{\omega}^{p}(\mathbb{D}),
 1\leq p\leq \infty,$ and $g=D_{\nu}f, 1\leq \nu\leq d,$ then
 for any $D_{i_{1}}, 1\leq i_{1}\leq d,$ the following
inequality holds
\begin{equation}
\|D_{i_{1}}g\|_{p}\leq \omega\|g\|_{p}.\label{1step}
\end{equation}
If $f\in \mathbf{B}_{\omega}^{p}(\mathbb{D})$, then for any
$D_{\nu}, D_{i_{1}}, 1\leq j, \nu\leq d$ and $g=D_{\nu}f$ the
inequality
$$
\|D_{i_{1}}^{l}g\|_{p}=\|D_{i_{1}}^{l}D_{\nu}f\|_{p}\leq\omega^{l+1}\|f\|_{p}=
\omega^{l}\left(\omega\|f\|_{p}\right), l\in \mathbb{N},
$$
takes place. But then  for any $z\in \mathbb{C}$ we have

$$ \left\|e^{zD_{i_{1}}}g\right\|_{p}=
\left\|\sum
^{\infty}_{l=0}\left(z^{l}D_{i_{1}}^{l}g\right)/l!\right\|_{p}
\leq \omega\|f\|_{p}\sum
^{\infty}_{r=0}\frac{|z|^{l}\omega^{l}}{l!}= \omega
e^{|z|\omega}\|f\|_{p}, 1\leq i_{1}\leq d.
$$
As in the Lemma 2.1 it implies that for any functional $\psi^{*}$
on $ L_{p}(M), 1\leq p\leq \infty,$ the scalar function
$$
F(z)=\left<e^{zD_{i_{1}}}g, \psi^{*}\right>,1\leq i_{1}\leq d,
$$
is an entire function of exponential type $\omega$. Moreover,
since $e^{tD_{i_{1}}}$  is an
 isometry group in the space  $ L_{p}(M), 1\leq p\leq
\infty,$ this function $F(t)$ is
 bounded on the real line
$$
|F(t)|=|\left< e^{tD_{i_{1}}}g, \psi^{*}\right>|\leq\|\psi^{*}\|\|
e^{tD_{i_{1}}}g\|_{p}\leq\|\psi^{*}\|\|g\|_{p},
$$
for any functional $\psi^{*}$ on the space $L_{p}(M)$.

The same arguments which were used in the previous Lemma show the
identity
$$
\frac{d}{dt}F(t)=
\frac{d}{dt}\left<e^{tD_{i_{1}}}g,\psi^{*}\right>=
\left<\frac{d}{dt}e^{tD_{i_{1}}}g,\psi^{*}\right>=\left<
e^{tD_{i_{1}}}D_{i_{1}}g,\psi^{*}\right>,1\leq i_{1}\leq d.
$$
  An application of the classical
Bernstein inequality ( see \cite{A}, Ch. IV) to the function
$F(t)$ in the uniform norm on the real line gives the inequality
$$
\sup_{t\in \mathbb{R}}\left|\frac{d}{dt}F(t)\right|\leq \omega
\sup_{t\in \mathbb{R}}|F(t)|.
$$
In our notations it takes the form
$$
\sup_{t\in \mathbb{R}}\left|\left<
e^{tD_{i_{1}}}D_{i_{1}}g,\psi^{*}\right>\right|= \sup_{t\in
\mathbb{R}}\left |\frac{d}{dt}\left< e^{tD_{i_{1}}}g,
\psi^{*}\right>\right | \leq\omega\|\psi^{*}\|\|g\|_{p}.
$$
By selecting $t=0$ and a functional $\psi^{*}$ for which

$$
\left< D_{i_{1}}g,\psi^{*}\right>=\|D_{i_{1}}g\|, \|\psi^{*}\|=1,
$$
we obtain the inequality (\ref{1step}) for any $1\leq i_{1}\leq
d$. Now, suppose that we proved the inequality
\begin{equation}
\|D_{i_{k-1}}...D_{i_{1}}g\|_{p}\leq
\omega^{k-1}\|g\|_{p},\label{kstep} g=D_{\nu}f,
\end{equation}
for all $1\leq i_{1},...,i_{k-1}\leq d$. Then since $f\in
\mathbf{B}_{\omega}^{p}(\mathbb{D})$ for
$h=D_{i_{k-1}}...D_{i_{1}}g$ we will have
$$
\|D_{i_{k}}^{l}h\|_{p}=\|D_{i_{1}}^{l}D_{i_{k-1}}...D_{i_{1}}D_{\nu}f\|_{p}
\leq\omega^{l+k}\|f\|_{p}=
\omega^{l}\left(\omega^{k}\|f\|_{p}\right), l\in \mathbb{N}.
$$
At this point we can repeat all the previous arguments to obtain
$$
\|D_{i_{k}}h\|_{p}\leq \omega \|h\|_{p},
$$
which along with the induction assumption (\ref{kstep}) gives the
desired inequality
$$
\|D_{i_{k}}...D_{i_{1}}g\|_{p}=\|D_{i_{k}}h\|_{p}\leq \omega
\|h\|_{p}=\omega \|D_{i_{k-1}}...D_{i_{1}}g\|_{p}\leq
\omega^{k}\|g\|_{p}.
$$
 The
first part of the Theorem is proved.

To prove the second part of the Theorem it is enough to show that
\textit{a function $f$ belongs to the space
$\mathbf{B}_{\omega}^{p}(\mathbb{D}), 1\leq p\leq \infty,$ if and
only if  for any
  $1\leq i_{1},... ,i_{k}\leq d$,  any $1\leq j\leq d,$ and any functional
   $\psi^{*}\in  L_{p}(M)^{*}$ the function
\begin{equation}
\left<\psi^{*},e^{tD_{j}}D_{i_{1}}...D_{i_{k}}f\right>:
\mathbb{R}\rightarrow  \mathbb{R},\label{functional one}
\end{equation}
  of the real variable $t$ is an entire function of the exponential
  type $\omega$.}

 Suppose that $f\in
\mathbf{B}_{\omega}^{p}(\mathbb{D}), 1\leq p\leq \infty$, then for
any function $g=D_{i_{1}}...D_{i_{k}}f, 1\leq i_{1},...,i_{k}\leq
d,$ and any $1\leq j\leq d$ the series
\begin{equation}
e^{zD_{j}}g=\sum \frac{(zD_{j})^{r}}{r!}g\label{Func}
\end{equation}
 is convergent in $L_{p}(M)$ and represents
 an abstract entire function. Since
  $\|D_{j}^{r}g\|_{p}\leq
 \omega^{k+r}\|f\|_{p}$ we have the  estimate

$$ \left\|e^{zD_{j}}g\right\|_{p}=
\left\|\sum
^{\infty}_{r=0}\left(z^{r}D_{j}^{r}g\right)/r!\right\|_{p} \leq
\omega^{k}\|f\|_{p}\sum
^{\infty}_{r=0}\frac{|z|^{r}\omega^{r}}{r!}=
\omega^{k}e^{|z|\omega}\|f\|_{p},
$$
which shows that the function (\ref{Func}) has exponential type
$\omega$. Since $e^{tD_{j}}$ is a group of isometries, the
abstract function $e^{tD_{j}}g$ is bounded by
$\omega^{k}\|f\|_{p}$. It implies that for any functional
$\psi^{*}$ on $ L_{p}(M), 1\leq p\leq \infty,$ the scalar function
$$
F(z)=\left<\psi^{*}, e^{zD_{j}}g\right>
$$
is entire because it is defined by the series
\begin{equation}
F(z)=\left<\psi^{*}, e^{zD_{j}}g\right>=\sum^{\infty}_{r=0}
\frac{z^{r}\left<\psi^{*},D_{j}^{r}g\right>}{r!}
\end{equation}
and because $|\left<\psi^{*},D_{j}^{r}g\right>|\leq\omega^{k+r}
\|\psi^{*}\|\|f\|_{p}$ we have
\begin{equation}
|F(z)|\leq e^{|z|\omega}\omega^{k}\|\psi^{*}\|\|f\|_{p}.
\end{equation}
For real $t$ we also have
$|F(t)|\leq\omega^{k}\|\psi^{*}\|\|f\|_{p}.$ Thus, we proved the
if $f\in \mathbf{B}_{\omega}^{p}(\mathbb{D}), 1\leq p\leq \infty,$
then the function (\ref{functional one}) is an entire function of
the exponential type $\omega$.

To  prove the inverse  statement  let us note that the fact that
$f$ belongs to the space $ \mathbf{B}_{\omega}^{p}(\mathbb{D}),
1\leq p\leq \infty,$ means in particular that for
 any $1\leq j\leq d$ and any functional $\psi^{*}$ on $ L_{p}(M), 1\leq p\leq
\infty,$ the function $F(z)=\left<\psi^{*}, e^{zD_{j}}f\right>$ is
an entire function
 of exponential type $\omega$ which is bounded on the real axis  $\mathbb{R}^{1}$.
 Since $e^{tD_{j}}$ is a group of isometries in $ L_{p}(M)$, an application of
  the Bernstein inequality for functions of one variable gives
$$ \left\|\left<\psi^{*},
e^{tD_{j}}D_{j}^{m}f\right>\right\|_{C(R^{1})}=\left
\|\left(\frac{d}{dt}\right)^{m}\left<\psi^{*},
e^{tD_{j}}f\right>\right \|_{ C(R^{1})}\leq\omega^{m}\|\psi^{*}\|
\|f\|_{p}, m\in\mathbb{N}.
$$

The last one gives for $t=0$

$$ \left|\left<\psi^{*}, D_{j}^{m}f\right>\right|\leq \omega^{m} \|\psi^{*}\|
 \|f\|_{p}.$$

Choosing $h$ such that $\|\psi^{*}\|=1$ and
\begin{equation}
\left<\psi^{*}, D_{j}^{m}f\right>=\|D_{j}^{m}f\|_{p}
\end{equation}
we obtain the inequality
\begin{equation}
\|D_{j}^{m}f\|_{p}\leq \omega^{m}\|f\|_{p}, m\in \mathbb{N}.
\end{equation}
It was the first step of induction. Now assume that we already
proved that the fact that $f$ belongs to the space
$\mathbf{B}_{\omega}^{p}(\mathbb{D}), 1\leq p\leq \infty,$ implies
the inequality
$$
\|D_{i_{1}}...D_{i_{k}}f\|_{p}\leq \omega^{k}\|f\|_{p}
$$
for any choice of indices $1\leq i_{1},i_{2}...,i_{k}\leq d$. Then
we can apply our first step of induction to the function
$g=D_{i_{1}}...D_{i_{k}}$. It proves  if  for any
  $1\leq i_{1},... ,i_{k}\leq d$,  any $1\leq j\leq d,$ and any functional
   $\psi^{*}\in  L_{p}(M)^{*}$ the function
\begin{equation}
\left<\psi^{*},e^{tD_{j}}D_{i_{1}}...D_{i_{k}}f\right>:
\mathbb{R}\rightarrow  \mathbb{R},\label{functional}
\end{equation}
  of the real variable $t$ is an entire function of the exponential
  type $\omega$ then $f\in \mathbf{B}_{\omega}^{p}(\mathbb{D}), 1\leq p\leq \infty.$
Thus the second part of the Theorem 2.2 is proved.

In order to prove the part 3 of the Theorem 2.2 we assume that a
sequence $f_{k}\in \mathbf{B}_{\omega}^{p}(\mathbb{D}), 1\leq
p\leq \infty,$ converges in $L_{p}(M)$ to a function $f$. Because
of the Bernstein inequality for any $1\leq j\leq d$ the sequence
$D_{j}f_{k}$ will be fundamental in $L_{p}(M)$. Note, that since
the operator $D_{j}$ is a generator of a strongly continuous group
of operators in the space $L_{p}(M)$ it is closed (\cite{BB}, Ch.
1). It shows that the limit of the sequence $D_{j}f_{k}$ is  the
function $D_{j}f$ and because of it the following  inequality
holds
$$
\|D_{j}f\|_{p}\leq \omega\|f\|_{p}.
$$
By repeating these arguments we can show that if a sequence
$f_{k}\in \mathbf{B}_{\omega}^{p}(\mathbb{D})$ converges in
$L_{p}(M)$ to a function $f$ then the Bernstein inequality
$$
\|D_{i_{k}}...D_{i_{1}}f\|_{p}\leq \omega^{k}\|f\|_{p}, 1\leq
p\leq \infty,
$$
for $f$ holds true. The Theorem 2.2 is proved.

\end{proof}

Consider a compact symmetric space $M=G/K$, where $G$ is a compact
Lie group. It is known (\cite{Z}, Ch. IV) that every compact Lie
group can be considered as a closed subgroup of the orthogonal
group $O(\mathbb{R}^{N})$ of a certain Euclidean space
$\mathbb{R}^{N}$. This fact allows to identify $M$ with an orbit
of a unit vector $v\in \mathbb{R}^{N}$ under action of a subgroup
of the orthogonal group $O(\mathbb{R}^{N})$ in $\mathbb{R}^{N}$.
In this case $K$ will be  the stationary group of $v$. Such
embedding of $M$ into $\mathbb{R}^{N}$ is called equivariant.

We choose an orthonormal  basis in $\mathbb{R}^{N}$ for which the
first vector is the vector $v$: $e_{1}=v, e_{2},...,e_{N}$. Let $
\textbf{P}_{n}(M) $ be the space of restrictions to $M$ of all
polynomials in $\mathbb{R}^{N}$ of degree $n$. This space is
closed in the norm of $L_{p}(M), 1\leq p\leq \infty,$ which is
constructed with respect to the $G$-invariant measure on $M$.

Let $T$ be the quasi-regular representation of $G$ in the space
$L_{p}(M), 1\leq p\leq \infty$ . In other words, if $ f\in
L_{p}(M), g\in G, x\in M$, then
$$
\left(T(g)f\right)(x)=f(g^{-1}x).
$$
Lie algebra $\textbf{g}$ of the group $G$ is formed by $N\times N$
skew-symmetric matrices $X$ for which $\exp tX\in  G$ for all
$t\in \mathbb{R}$. The scalar product in $\textbf{g}$ is given by
the formula
$$
<X_{1},X_{2}>=\frac{1}{2}tr(X_{1}X_{2}^{t})=-\frac{1}{2}tr(X_{1}X_{2}),
X_{1},X_{2}\in \textbf{g}.
$$
Let $X_{1},X_{2},...,X_{d}$ be an orthonormal basis of
$\textbf{g}, \dim \textbf{g}=d,$ and $D_{1}, D_{2}, ...,D_{d}$ be
the corresponding infinitesimal operators of the quasi-regular
representation of $G$ in $L_{p}(M), 1\leq p\leq\infty$.
\begin{thm} If $M$ is equivariantly embedded into $\mathbb{R}^{N}$
then for any $1\leq p\leq \infty$ the following inclusion holds
true
\begin{equation}
\textbf{P}_{n}(M)\subset
\mathbf{B}_{\omega}^{p}(\mathbb{D}).\label{P-B}
\end{equation}

\end{thm}
\begin{proof}
The general theory  of skew-symmetric matrices (see \cite{GL}, Ch.
3) implies that for any skew symmetric matrix $X$ each element of
the matrix $\exp tX$ is a linear combination of the functions
$\cos t\theta_{i}, \sin t \theta_{i} $ for certain real numbers
$\theta_{1},...,\theta_{[n/2]},$
 for which
 $$
 \|X\|^{2}=\sum_{i=1}^{[N/2]}\theta_{i}^{2}.
 $$
It shows that for any point $x\in M$ and any basis vector $e_{i},
1\leq i\leq N,$ in $\mathbb{R}^{N}$ the coordinate function
$$
x_{i}(\exp tX\cdot x)=\left<\exp tX\cdot x, e_{i}\right>
$$
 is also a linear
combination of $\cos t\theta_{i}, \sin t\theta_{i}$ whose
coefficients are smooth functions of $x$. Thus we conclude that
every function  $x_{i}(\exp tX\cdot x)$ has extension to the
complex plane $\mathbb{C}$ as entire function of exponential type
$\leq \max|\theta_{i}|\leq \|X\|$.

Since for every $f\in  \textbf{P}_{n}(M) $ and every basis vector
$X_{j}, j=1,...,d; \|X_{j}\|=1, $ the function $f(\exp tX_{j}\cdot
x), x\in M,$ is a polynomial of degree $n$ in $x_{i}(\exp
tX_{j}\cdot x)$ we obtain that for any functional $\psi^{*}$ on
$L_{p}(M), 1\leq p\leq \infty,$ the function
$$
F_{X_{j}}(t)=\left<f(\exp tX_{j}\cdot x), \psi^{*}\right>, x\in
M,\|X_{j}\|=1,
$$
has exponential type at most $n$ and according to the classical
Riesz interpolation formula it  gives
\begin{equation}
\frac{dF_{X_{j}}(t)}{dt}=
\frac{n}{\pi^{2}}\sum_{k\in\mathbb{Z}}\frac{(-1)^{k-1}}{(k-1/2)^{2}}
F_{X_{j}}\left(t+\frac{\pi}{n}\left(k-1/2\right)\right), t\in
\mathbb{R}.
\end{equation}

Since
$$
 \frac{d F_{X_{j}}(t)}{dt}=\left<\frac{d}{dt}f(\exp tX_{j}\cdot x), \psi^{*}\right>
 =\left<D_{j}f(\exp
tX_{j}\cdot x), \psi^{*}\right>
$$
we have
$$
\left<D_{j}f(\exp tX_{j}\cdot x), \psi^{*}\right>=
\frac{n}{\pi^{2}}\sum_{k\in\mathbb{Z}}\frac{(-1)^{k-1}}{(k-1/2)^{2}}
\left<f\left(\exp\left(t+\frac{\pi}{n}\left(k-1/2\right)X_{j}\cdot
x\right)\right),\psi^{*}\right>,
$$
where $t\in \mathbb{R}, x\in M$. For $t=0$ it gives
$$
\left<D_{j}f( x), \psi^{*}\right>=
\frac{n}{\pi^{2}}\sum_{k\in\mathbb{Z}}\frac{(-1)^{k-1}}{(k-1/2)^{2}}
\left<f\left(\exp\left(\frac{\pi}{n}\left(k-1/2\right)X_{j}\cdot
x\right)\right),\psi^{*}\right>.
$$
Because this formula holds true for any functional $\psi^{*}$ it
implies
$$
D_{j}f( x)=
\frac{n}{\pi^{2}}\sum_{k\in\mathbb{Z}}\frac{(-1)^{k-1}}{(k-1/2)^{2}}
f\left(\exp\left(\frac{\pi}{n}\left(k-1/2\right)X_{j}\cdot
x\right)\right), x\in M.
$$
Since
$$
e^{tD_{X_{j}}}f(x)=f(\exp tX_{j}\cdot x), t\in \mathbb{R}, f\in
L_{p}(M), x\in M,
$$
it gives
$$
D_{j}f( x)=
\frac{n}{\pi^{2}}\sum_{k\in\mathbb{Z}}\frac{(-1)^{k-1}}{(k-1/2)^{2}}
e^{t_{k}D_{j}}f( x), t_{k}=\frac{\pi}{n}\left(k-1/2\right), x\in
M.
$$
Because
$$
\frac{n}{\pi^{2}}\sum_{k\in\mathbb{Z}}\frac{(-1)^{k-1}}{(k-1/2)^{2}}=n,
$$
and because $\|e^{t_{k}D_{j}}f\|_{p}=\|f\|_{p},$ we obtain the
Bernstein inequality
\begin{equation}
\|D_{j}f\|_{p}\leq n\|f\|_{p}, 1\leq p\leq \infty.\label{Bern1}
\end{equation}
Since the set of polynomials $\textbf{P}_{n}(M) $ is invariant
under translations and closed  \cite{R} it is invariant under all
operators $D_{1},D_{2},...,D_{d}, d=dim G.$ It is clear that  by
using invariance of $\textbf{P}_{n}(M) $ and the inequality
(\ref{Bern1}) we obtain  the desired inequality
$$
\|D_{j_{1}}...D_{j_{k}}f\|_{p}\leq n^{k}\|f\|_{p}, k\in
\mathbb{N}.
$$
The Theorem is proved.
\end{proof}

As a consequence of this Theorem and the Lemma 2.1 we obtain the
following Corollary.

\begin{col}
If $M$ is equivariantly embedded into $\mathbb{R}^{N}$ then for
any polynomial $f\in \textbf{P}_{n}(M)$ the following Riesz
interpolation formula holds
$$
D_{j_{1}}D_{j_{2}}...D_{j_{k}}f(
x)=\mathcal{R}_{j_{1}}^{n}\mathcal{R}_{j_{2}}^{n}...\mathcal{R}_{j_{k}}^{n}f(x),
$$
where
$$
\mathcal{R}_{j}^{n}f(x)=\frac{n}{\pi^{2}}\sum_{k\in\mathbb{Z}}\frac{(-1)^{k-1}}{(k-1/2)^{2}}
e^{t_{k}D_{j}}f( x), t_{k}=\frac{\pi}{n}\left(k-1/2\right), x\in
M, 1\leq j\leq d.
$$
\end{col}

 The next interpolation inequality for generators of
one-parameter strongly continuous groups of operators in Banach
spaces will be used in the following section.

\begin{lem}
  If $A$ generates a $C_{0}$-one-parameter group of operators
$e^{tA}$ such that $\|e^{tA}f\|=\|f\|$ then for every $n\geq 2$
there exists a $C(n)$ such that for all $\varepsilon >0$ all
$1\leq m \leq n-1$ and all $f$ in the domain of $A^{n}$

\begin{equation}
\|A^{m}f\|\leq \varepsilon ^{n-m}\|A^{n}f\|+\varepsilon
^{-m}C(n)\|f\|.\label{int}
\end{equation}
\end{lem}
\begin{proof}
  According to the Hille-Phillips-Yosida theorem (\cite{BB}, Ch. I),
  the assumptions of the Lemma imply
$$
 \|(I+\varepsilon A)^{-1}\|\leq 1
 $$
 and the same for the operator
$(I-\varepsilon A)$.  Then
$$
\|f\|\leq \|(I+\varepsilon A)f\|
$$
 and the same for the operator
$(I-\varepsilon A)$.  It gives

$$\varepsilon \|Af\| \leq \|(I-\varepsilon A)f\|+\|f\| \leq
\|(I+\varepsilon ^{2}A^{2})f\|+\|f\|\leq
\varepsilon^{2}\|A^{2}f\|+2\|f\|.$$

So, for any $f$ from the domain of $A^{2}$ we have inequality

$$
\|Af\|\leq \varepsilon \|A^{2}f\|+2/\varepsilon \|f\|, \varepsilon
> 0 .
$$

The general case can be proved by induction.
\end{proof}

\section{Two Nikolskii-type inequalities on compact homogeneous
manifolds}

Let $M=G/K$, dim$M=n, dim G=d, $ be a homogeneous manifold which
we consider with invariant Riemannian metric and corresponding
Riemannian measure $dx$. The $ B(x,\rho)$ will denote a ball whose
center is $x\in M$ and radius is $\rho >0$. Denote by $T_{x}(M)$
the tangent space of $M$ at a point $x\in M$ and let $ exp_ {x} $
:  $T_{x}(M)\rightarrow M$ be the exponential geodesic map i.  e.
$exp_{x}(u)=\gamma (1), u\in T_{x}(M)$ where $\gamma (t)$ is the
geodesic starting at $x$ with the initial vector $u$ :  $\gamma
(0)=x , \frac{d\gamma (0)}{dt}=u.$ Since the manifold $M$ is
compact there exists a positive $\rho_{M}$ such that the
exponential map is a diffeomorphism of a ball of radius
$\rho<\rho_{M} $ in the tangent space $T_{x}(M)$ onto the ball
$B(x , \rho )$ for every $x\in M$.  We consider only coordinate
systems on $M$ which are given by the exponential map.

We fix a  cover $B=\{B(y_{\nu}, r_{0})\}$ of $M$ of finite
multiplicity $N(M)$ (see Lemma 1.1)
\begin{equation}
M=\bigcup B(y_{\nu}, r_{0}),\label{cover}
\end{equation}
where $B(y_{\nu}, r_{0})$ is a  ball at $y_{\nu}\in M$ of radius
$r_{0}\leq \rho_{M},$ and consider a fixed partition of unity
$\Psi=\{\psi_{\nu}\}$ subordinate to this cover. The Sobolev
spaces $W^{k}_{p}(M), k\in \mathbb{N}, 1\leq p<\infty,$ are
introduced as the completion of $C_{0}^{\infty}(M)$ with respect
to the norm
\begin{equation}
\|f\|_{W^{k}_{p}(M)}=\left(\sum_{\nu}\|\psi_{\nu}f\|^{p}
_{W^{k}_{p}(B(y_{\nu}, r_{0}))}\right) ^{1/p}.\label{Sobnorm}
\end{equation}
Any two such norms are equivalent. We consider the system of
vector fields $\mathbb{D}=\{X_{1},...,X_{d}\}, d=dim G, $ on
$M=G/K,$ which was described in the Introduction. Since vector
fields $\mathbb{D}=\{X_{1},...,X_{d}\}$ generate the tangent space
at every point of $M$ and $M$ is compact it is clear that the
Sobolev norm (\ref{Sobnorm}) is equivalent to the norm
\begin{equation}
\|f\|_{p}+\sum_{j=1}^{k} \sum_{1\leq i_{1},...,i_{j}\leq
d}\|D_{i_{1}}...D_{i_{j}}f\|_{p}, 1\leq p\leq
\infty.\label{mixednorm}
\end{equation}
 Using the closed graph Theorem and
the fact that every $D_{i}$ is a closed operator in $L_{p}(M),
1\leq p<\infty,$ it is easy to show that the norm
(\ref{mixednorm}) is equivalent to the norm
\begin{equation}
|||f|||_{k,p}=\|f\|_{p}+\sum_{1\leq i_{1},..., i_{k}\leq
d}\|D_{i_{1}}...D_{i_{k}}f\|_{p}, 1\leq p\leq
\infty.\label{mixednorm2}
\end{equation}
In other words, there exist constants $c_{0}(\mathbb{D},B,\Psi,k),
C_{0}(\mathbb{D},B,\Psi, k)$ such that
\begin{equation}
c_{0}(\mathbb{D},B,\Psi, k)\|f\|_{W^{k}_{p}(M)}\leq |||f|||_{k,p}
 \leq C_{0}(\mathbb{D},B,\Psi, k)\|f\|_{W^{k}_{p}(M)}.
\end{equation}
 Since the Laplace operator $\mathcal{L}$ which is defined in (1.7)
 is an elliptic operator on a compact manifold $M$
 the regularity theorem for $\mathcal{L}$ means in
particular (\cite{T}, Ch. I and Ch. III), that the norm of the
Sobolev space $W_{p}^{2k}(M), k\in \mathbb{N}, 1\leq p< \infty,$
is equivalent to the graph norm
$\|f\|_{p}+\|\mathcal{L}^{k}f\|_{p}$. Thus, there exist two
constants $c_{1}(\mathcal{L},B,\Psi,k), C_{1}(\mathcal{L},B,
\Psi,k)$ such that
\begin{equation}
c_{1}(\mathcal{L},B, \Psi,k)\|f\|_{W^{2k}_{p}(M)}\leq
\|f\|_{p}+\|\mathcal{L}^{k}f\|_{p}\leq C_{1}(\mathcal{L},B,
\Psi,k)\|f\|_{W^{2k}_{p}(M)}.
\end{equation}

In what follows $o\in M$ will denote  the "origin" of the
homogeneous manifold $M$ which corresponds to the coset defined by
the subgroup $K$ in the representation $M=G/K$. Note that since
$G$ acts on $M$, every function on $M$ can be treated as a
function on $G$ according to the formula
$$
f(g)=f(g\cdot o), g\in G.
$$
\begin{thm}

 For any   $1\leq p\leq \infty,$ any natural $l>m/p, m=dim
 M,$ there exists a constant $C(M,l)$ such that for any
 $(r,N(G))$-lattice $Z_{G}(g_{\nu}, r,
N(G))\subset
 G$ with sufficiently small $r>0$,
 any $\omega>0$ and any $q\geq p$ the following
 inequalities hold true
\begin{equation}
 \|f\|_{q}\leq   r^{m/q}\sup_{g\in
G}\left(\sum_{g_{i}\in Z_{G}(r, N(G))}\left(|f(gg_{i}\cdot
o)|\right)^{p}\right)^{1/p}\leq
$$
$$
C(M, l)r^{m/q-m/p}\left(1+(r\omega)^{l}\right)\|f\|_{p},\label{NI}
\end{equation}
for all $f\in  \mathbf{B}_{\omega}^{p}(\mathbb{D}) $.

\end{thm}

\begin{proof}

 Our
nearest goal is to prove the right-hand side of the inequality
(\ref{NI}). We fix a sufficiently small  ball $B(o, r),
0<r<r_{0},$ in the tangent space $T_{o}M$  at the  origin $o\in M$
and an $(r, N(M))$-lattice $\{g_{i}\}=Z_{G}(r, N(M))\subset G,$
such that
 translations $g_{i}\cdot B(o,
r)=B(x_{i}, r), x_{i}=g_{i}\cdot o,$ of the ball $B(o, r)$ are
disjoint.
 We are going to use the following form of the
Sobolev inequality (see \cite{Ad}, Lemma 5.15, or \cite{Nar},
Corollary 3.5.12)
\begin{equation}
|\phi(x)|\leq C_{1}(m,l)\sum_{0\leq j\leq
l}r^{j-m/p}\|\phi\|_{W_{p}^{j}(B(x_{i}, r))}, l>m/p,\label{Sob}
\end{equation}
  where
$x\in B(x_{i}, r/2), \phi \in C^{\infty}\left(B(x_{i}, r)\right)$.

 We apply the
inequality (\ref{Sob}) to a function $\psi_{\nu}f$ from the
formula  (\ref{Sobnorm}) which gives the Sobolev norm. Since
$\psi_{\nu}$ is a partition of unity and since $N(M)$ is a number
 of balls in the cover $B(y_{\nu}, r_{0})$ which intersect each other we
obtain for $x_{i}\in B(y_{\nu}, r_{0})$
\begin{equation}
r^{m}|f(x_{i})|^{p}=r^{m}\left|\sum_{\nu}\psi_{\nu}f(x_{i})\right|^{p}\leq
$$
$$
(N(M))^{p}\sum_{\nu}\left(r^{m/p}|\psi_{\nu}f(x_{i})|\right)^{p}\leq
C_{1}(m,l)(N(M))^{p}\sum_{\nu}\sum_{0\leq j\leq
l}r^{jp}\|\psi_{\nu}f\|_{W_{p}^{j}(B(x_{i}, r))}^{p},
\end{equation}
where $ l>m/p,1\leq p\leq\infty.$  Summation over $i$  gives the inequality
$$
\sum_{i}\left(r^{m/p}|f(x_{i})|\right)^{p}\leq C_{1}(m,l)(
N(M))^{p}\sum_{\nu}\sum_{0\leq j\leq
l}r^{jp}\sum_{i}\|\psi_{\nu}f\|_{W_{p}^{j}(B(x_{i}, r))}^{p},
l>m/p.
$$
 Since  the balls $B(x_{i}, r)$ are disjoint and the
support of $\psi_{\nu}$ is a subset of $B(y_{\nu}, r_{0})$, we
obviously have
\begin{equation}
\sum_{i}\|\psi_{\nu}f\|_{W_{p}^{j}(B(x_{i}, r))}^{p}\leq
\|\psi_{\nu}f\|_{W_{p}^{j}(B(y_{\nu}, r_{0}))}^{p}, 0\leq  j\leq
l.
\end{equation}
 Thus we obtain
that for any given $l>m/p $  there exists a constant
 $C_{5}(M,l)>0$, such that for any $(r,N(M))$-lattice $\{g_{i}\}=Z_{G}(r, N(M))\subset G,$
  the following inequality holds true for $1\leq
p\leq\infty$

$$
\left(\sum_{i}\left(r^{m/p}|f(x_{i})|\right)^{p}\right)^{1/p}\leq
C_{5}(M,l)
N(M)\left(\|f\|_{p}^{p}+\sum_{j=1}^{l}r^{jp}\sum_{\nu}\|\psi_{\nu}f\|^{p}_{W_{p}^{j}(B(y_{\nu},
r_{0}))}\right)^{1/p}\leq
$$
$$
C_{5}(M,l)N(M)\left(\|f\|_{p}^{p}+\sum_{j=1}^{l}r^{jp}\|\psi_{\nu}f\|^{p}_{W_{p}^{j}(M)}\right)^{1/p}.
$$

Since the norms (\ref{Sobnorm}) and (\ref{mixednorm2}) are
equivalent we obtain the inequality

\begin{equation}
\left(\sum_{i}\left(r^{m/p}|f(x_{i})|\right)^{p}\right)^{1/p}\leq
$$
$$
C_{6}\left(\|f\|_{p}+\sum_{j=1}^{l} \sum_{0\leq
k_{1},...,k_{j}\leq d}r^{j}\|D_{k_{1}}...D_{k_{j}}f\|_{p}\right),
l>m/p,
\end{equation}
where $ C_{6}=C_{6}(M,\mathbb{D},B,\Psi, l, N(M)).$ Because every
$D_{k}, k=1,...,d,$ is a generator of a one-parameter isometric
group of bounded operators in $L_{p}(M)$, the interpolation
inequality (\ref{int}) can be used and then the last two
inequalities imply the following estimate
$$
\left(\sum_{i}(r^{m/p}|f(x_{i})|)^{p}\right)^{1/p}\leq
$$
$$
C_{7}\left(\|f\|_{p}+r^{l} \sum_{0\leq k_{1},...,k_{l}\leq
d}\|D_{k_{1}}...D_{k_{l}}f\|_{p}\right), l>m/p,
$$
where $C_{7}=C_{7}(M,\mathbb{D},B,\Psi, l, N(M))$.
 For $f\in \mathbf{B}_{\omega}^{p}(\mathbb{D})$ it gives
$$
\left(\sum_{i}\left(r^{m/p}|f(x_{i})|\right)^{p}\right)^{1/p}\leq
C_{8}\left(1+(r\omega)^{l}\right)\|f\|_{p}, l>m/p,
$$
where  $C_{8}=C_{8}(M,\mathbb{D},B,\Psi, l, N(M)).$ Applying this
inequality to a translated function $f(h\cdot x), h\in G,$ and
using invariance of the measure $dx$ we obtain for $f\in
\mathbf{B}_{\omega}^{p}(\mathbb{D})  $

$$
\sup_{h\in G}\left(\sum_{i}\left(r^{m/p}\left|f(h\cdot
x_{i})\right|\right)^{p}\right) ^{1/p}\leq
$$
\begin{equation}
 C_{8}\left(1+(r\omega)^{l}\right)\sup_{h\in
G}\left\|f(h\cdot x)\right\|_{p}=
C_{8}\left(1+(r\omega)^{l}\right)\|f\|_{p}, l>m/p.
\end{equation}
This inequality implies the right-hand side of the inequality
(\ref{NI}).

 To prove the left-hand side of the (\ref{NI}) we introduce the following neighborhood of the
identity in the group $G$
$$
Q_{4r}=\left\{g\in G : g\cdot o\in B(o,4r)\right\}.
$$
According to the following formula which holds true for any
continuous function $f$  on $M$

$$
\int_{M}f(x)dx=\int_{G}f(g\cdot o)dg,
$$
we have the following estimate for the characteristic function
$\chi_{B}$ of the ball $B(o,4r)$
$$
(4r)^{m}\approx\int_{B(o,4r)}dx=
\int_{M}\chi_{B}(x)dx=\int_{G}\chi_{B}(g\cdot
o)dg=\int_{Q_{4r}}dg.
$$
Since every ball in our cover is a translation of the fixed ball
$B(o, 4r)$ and these balls form a cover of $M$ the $G$-invariance
of the measure $dx$ gives

$$\int_{M}|f(x)|^{q}dx\leq
\sum_{g_{i}\in Z_{G}(r, N(G))}\int_{g_{i}\cdot
B(o,4r)}|f(x)|^{q}dx\leq\sum_{g_{i}\in Z_{G}(r,
N(G))}\int_{B(o,4r)}|f(g_{i}\cdot y)|^{q}dy=
$$
$$
\int_{Q_{4r}}\sum_{g_{i}\in Z_{G}(r, N(G))}|f(g_{i}h\cdot
o)|^{q}dh\leq (4r)^{m}\sup_{g\in G}\sum_{g_{i}\in Z_{G}(r,
N(G))}|f(g_{i}g\cdot o)|^{q},
$$
where $f\in L_{q}(M), 1\leq q\leq \infty,$ $m=dim M$. Next, using
the inequality
$$
\left(\sum a_{i}^{q}\right)^{1/q}\leq \left(\sum
a_{i}^{p}\right)^{1/p},
$$
which holds true for any $ a_{i}\geq 0, 1\leq p\leq q\leq \infty,$
we obtain the following inequality

\begin{equation}
\|f\|_{q}\leq 4^{m}r^{m/q} \sup_{g\in G}\left(\sum_{g_{i}\in
Z_{G}(r, N(G))}\left(|f(g_{i}g\cdot
o)|\right)^{q}\right)^{1/q}\leq
$$
$$
4^{m}r^{m/q} \sup_{g\in G}\left(\sum_{g_{i}\in Z_{G}(r,
N(G))}\left(|f(g_{i}g\cdot o)|\right)^{p}\right)^{1/p}=
$$
$$
4^{m}r^{m/q-m/p} \sup_{g\in G}\left(\sum_{g_{i}\in Z_{G}(r,
N(G))}\left(r^{m/p}|f(g_{i}g\cdot o)|\right)^{p}\right)^{1/p} .
\end{equation}

From these
 inequalities  and the observation, that
 for the element $ g=g_{i}^{-1}hg_{i} $ the expression
$$
\sum_{g_{i}\in Z_{G}(r, N(G))}\left(r^{m/p}\left|f(g_{i}g\cdot
o)\right|\right)^{p}
$$
becomes the expression
$$\sum_{g_{i}\in Z_{G}(r, N(G))}\left(r^{m/p}\left|f(hg_{i}\cdot
o)\right|\right)^{p},
$$
we obtain the left-hand side of the inequality (\ref{NI}). The
Theorem 3.1 is proved.

\end{proof}

This Theorem is used to prove the following result.

\begin{thm}
There exists a constant $C(M)$
 such that for any $1\leq p\leq q\leq \infty$
the following  inequality holds true for all $f\in
\mathbf{B}_{\omega}^{p}(\mathbb{D})$
\begin{equation}
\|f\|_{q}\leq
C(M)\omega^{\frac{m}{p}-\frac{m}{q}}\|f\|_{p},\label{Nik} m=dim M.
\end{equation}

\end{thm}
\begin{proof} The Theorem 3.1 imply that for  any $1\leq p\leq
\infty$, any natural
  $l>m/p$ there exists a constant $C(M,l)>0$  such that
 for any sufficiently small $r>0$, any $\omega>0$ and any
  $q\geq p$ the following inequality holds
 true
 $$
 \|f\|_{q}\leq C(M,l)r^{m/q-m/p}(1+(r\omega)^{l})\|f\|_{p},
 $$
 for all $f\in \mathbf{B}_{\omega}^{p}(\mathbb{D})$. We make the substitution
$t=r\omega$ into this inequality  to obtain
$$
\|f\|_{q}\leq C(M,l)\eta_{p,q}(t)\omega^{m/p-m/q}\|f\|_{p}, l>m/p,
$$
where
$$
\eta_{p,q}(t)=t^{m/q-m/p}(1+t^{l}), t\in(0,\infty), t=r\omega.
$$

Since $l$ can be any number greater than $m/p$ and $p\geq 1,$ we
fix the number $l=2m$. At the point
\begin{equation}
t_{m,p,q}=\frac{\alpha}{2m-\alpha}\in (0,1),\label{t}
\end{equation}
where
$0< \alpha=m/p-m/q< 1,$ the function $\eta_{p,q}$ has its minimum,
which is
$$
\eta_{p,q}(t_{m,p,q})=\frac{1}{(1-\beta)^{1-\beta}\beta^{\beta}}\leq
2,
$$
where $\beta=\alpha/2m$. For a given $m\in \mathbb{N}, \omega>0,
1\leq p\leq q\leq \infty,$ we can find corresponding $t_{m,p,q}$
using the formula (\ref{t}) and then can find the  corresponding
$r>0$ as $r=r_{m,p,q, \omega}=t_{m,p,q}/\omega$. For such $r$  one
can find a cover of the same multiplicity $N(M)$.  For this cover
we will have the inequality (\ref{Nik}). The Theorem is proved.

\end{proof}

\begin{thm}
For any $1\leq  p\leq q\leq \infty$ the following  equality holds
true
$$
\mathbf{B}_{\omega}^{p}(\mathbb{D})
=\mathbf{B}_{\omega}^{q}(\mathbb{D})\equiv
\mathbf{B}_{\omega}(\mathbb{D}).
$$
\end{thm}
\begin{proof}

First we show that
\begin{equation}
\mathbf{B}_{\omega}^{p}(\mathbb{D}) \subset
\mathbf{B}_{\omega}^{q}(\mathbb{D}), 1\leq  p\leq q\leq \infty
.\label{embed}
\end{equation}

 Since $\mathbf{B}_{\omega}^{p}(\mathbb{D}) $ is
invariant under every operator $D_{i}, 1\leq i\leq d,$ it is
enough to show that if $f\in \mathbf{B}_{\omega}^{p}(\mathbb{D})$,
then for any $1\leq j\leq d, k\in \mathbb{N},$
$$
\|D_{j}^{k}f\|_{q}\leq \omega^{k}\|f\|_{q}, 1\leq  p\leq q\leq
\infty.
$$
Because $f\in \mathbf{B}_{\omega}^{p}(\mathbb{D})$ and this set is
invariant under all operators $D_{i}$, the  Theorem 3.1 gives that
there exists a constant $C_{p,q}$ such that for any $z\in
\mathbb{C}$

$$ \left\|e^{zD_{j}}f\right\|_{q}=
\left\|\sum
^{\infty}_{l=0}\left(z^{l}D_{j}^{l}f\right)/l!\right\|_{q} \leq
C_{p,q}e^{|z|\omega}\|f\|_{p},
$$ It implies that for
any functional $\psi^{*}$ on $ L_{q}(M), 1\leq q\leq \infty,$ the
scalar function
$$
F(z)=\left< e^{zD_{j}}f, \psi^{*}\right>,
$$
is an entire function
 of exponential type $\omega$. At the same time it is bounded on the real axis
 $\mathbb{R}^{1}$
by the constant $\|\psi^{*}\| \|f\|_{q}$. The classical Bernstein
inequality gives
$$ \sup_{t\in \mathbb{R}}\left|\left<
e^{tD_{j}}D_{j}^{k}f, \psi^{*}\right>\right|=\sup_{t\in
\mathbb{R}}\left |\left(\frac{d}{dt}\right)^{k}\left< e^{tD_{j}}f,
\psi^{*}\right>\right |\leq\omega^{k}\|\psi^{*}\| \|f\|_{q},
m\in\mathbb{N}.
$$
When $t=0$ we obtain

$$ \left|\left<D_{j}^{k}f, \psi^{*}\right>\right|\leq \omega^{k} \|\psi^{*}\|
 \|f\|_{q}.$$
Choosing $\psi^{*}$ such that $\|\psi^{*}\|=1$ and
$$
\left< D_{j}^{k}f, \psi^{*}\right>=\|D_{j}^{k}f\|_{q}
$$
we get the inequality
$$
\|D_{j}^{k}f\|_{q}\leq \omega^{k}\|f\|_{q}, k\in \mathbb{N}.
$$

To prove an embedding which is  opposite to (\ref{embed}) we use
the fact that $M$ is compact and because of this the
$L_{\infty}(M)$-norm dominates any $L_{p}(M)$-norm with $1\leq
p<\infty$.  It gives the following inequality for any $f\in
\mathbf{B}_{\omega}^{\infty}(\mathbb{D}),1\leq p\leq \infty$

$$ \left\|e^{zD_{j}}f\right\|_{p}=
\left\|\sum
^{\infty}_{l=0}\left(z^{l}D_{j}^{l}f\right)/l!\right\|_{p} \leq
e^{|z|\omega}\|f\|_{\infty},
$$
which implies that for any functional $\psi^{*}$ on $ L_{p}(M),
1\leq p\leq \infty,$ the scalar function
$$
F(z)=\left< e^{zD_{j}}f, \psi^{*}\right>,
$$
is an entire function
 of exponential type $\omega$ which is bounded on the real axis  $\mathbb{R}^{1}$
by the constant $\|\psi^{*}\| \|f\|_{p}$. At this point we can use
the same arguments which were used  above. The Theorem is proved.

\end{proof}

\section{Relations between $\mathbf{B}_{\omega}(\mathbb{D})$,
$\textbf{E}_{\omega}(\mathcal{L})$ and $\textbf{P}_{n}(M)$}

We keep the same notations as above.

\begin{thm}The following equality takes place
\begin{equation}
\|\mathcal{L}^{k/2}f\|_{2}^{2}=\sum_{1\leq i_{1},...,i_{k}\leq
d}\|D_{i_{1}}...D_{i_{k}}f\|_{2}^{2},\label{eq0}
\end{equation}
which implies the following embeddings
$$
\mathbf{B}_{\sqrt{\omega/d}}(\mathbb{D})\subset
\textbf{E}_{\omega}(\mathcal{L})=
\mathbf{B}_{\omega}^{2}(\mathcal{L})\subset
\mathbf{B}_{\sqrt{\omega}}(\mathbb{D}),
$$
and in particular the following equality
$$
\bigcup_{\omega} \textbf{B}_{\omega}(\mathbb{D})=\bigcup_{j}
\textbf{E}_{\lambda_{j}}(\mathcal{L}).
$$

\end{thm}
\begin{proof}
Since the spectrum of $\mathcal{L}$ is discrete and of finite
multiplicity, the space $\textbf{E}_{\omega}(\mathcal{L})$ is
finite dimensional and the norm of $\mathcal{L}$ on this space is
exactly $\omega$. It gives the embedding
\begin{equation}
\textbf{E}_{\omega}(\mathcal{L})\subset
\mathbf{B}_{\omega}^{2}(\mathcal{L}).\label{eq1}
\end{equation}

Conversely, let  $0=\lambda_{0}<\lambda_{1}\leq \lambda_{2}\leq
... $ be the set of eigenvalues  of $\mathcal{L}$ listed with
multiplicities and $\varphi_{0}, \varphi_{1}, \varphi_{2}, ...$ be
a corresponding complete system of orthonormal eigenfunctions.
Assume that
$$
\lambda_{m}\leq \omega<\lambda_{m+1}.
$$
If a function $f$ belongs to the space
$\mathbf{B}_{\omega}^{2}(\mathcal{L})$ and the Fourier series
\begin{equation}
f=\sum_{j=0}^{\infty}c_{j}\varphi_{j}
\end{equation}
contains terms with $j\geq \lambda_{m+1}$, then
$$
\lambda_{m+1}^{2k}\sum_{j=m+1}^{\infty}|c_{j}|^{2}\leq
\sum_{j=m+1}^{\infty}|\lambda_{j}^{k}c_{j}|^{2}\leq
\|\mathcal{L}^{k}f\|^{2}\leq \omega^{2k}\|f\|^{2},
$$
which implies
$$
\sum_{j=m+1}^{\infty}|c_{j}|^{2}\leq
\left(\frac{\omega}{\lambda_{m+1}} \right)^{2k}\|f\|^{2}.
$$
In the last inequality the fraction $\omega/\lambda_{m+1}$ is
strictly less than $1$  and $k$ can be any natural number. It
shows that the series (4.3) does not contain terms with $j\geq
m+1$, i.e.  function $f$ belongs to $
\textbf{E}_{\omega}(\mathcal{L})$. We proved the inclusion
$$
\mathbf{B}_{\omega}^{2}(\mathcal{L})\subset\textbf{E}_{\omega}(\mathcal{L}),
$$
which gives along with (\ref{eq1}) the equality
\begin{equation}
\mathbf{B}_{\omega}^{2}(\mathcal{L})=\textbf{E}_{\omega}(\mathcal{L}).\label{eq00}
\end{equation}
The operator
$$
-\mathcal{L}=D_{1}^{2}+...+D_{d}^{2}
$$
commutes with every $D_{j}$ (see the explanation before the
formula (\ref{Laplacian}) in the Introduction).
 The same is
true for $\mathcal{L}^{1/2}$. But then
$$
\|\mathcal{L}^{1/2}f\|_{2}^{2}=<\mathcal{L}^{1/2}f,\mathcal{L}^{1/2}f>=<\mathcal{L}
f,f>=
$$
$$
-\sum_{j=1}^{d}<D_{j}^{2}f,f>=\sum_{j=1}^{d}<D_{j}f,D_{j}f>=
\sum_{j=1}^{d}\|D_{j}f\|_{2}^{2},
$$
$$
\|\mathcal{L}f\|_{2}^{2}=\|\mathcal{L}^{1/2}\mathcal{L}^{1/2}f\|_{2}^{2}=
\sum_{j=1}^{d}\|D_{j}\mathcal{L}^{1/2}f\|_{2}^{2}=
$$
$$
\sum_{j=1}^{d}\|\mathcal{L}^{1/2}D_{j}f\|_{2}^{2}=\sum_{j,k=1}^{d}\|D_{j}D_{k}f\|_{2}^{2}.
$$
From here by induction on $k$ we obtain (4.1). It proves the
formula (\ref{eq0}) which implies the rest of the Theorem. Indeed,
if $f\in \mathbf{B}_{\omega}^{2}(\mathcal{L})$ we obtain that
$f\in
\mathbf{B}_{\omega}^{2}(\mathbb{D})=\mathbf{B}_{\omega}(\mathbb{D})
$ because
$$
\|D_{i_{1}}...D_{i_{k}}f\|_{2}\leq \left(\sum_{1\leq
i_{1},...,i_{k}\leq
d}\|D_{i_{1}}...D_{i_{k}}f\|_{2}^{2}\right)^{1/2}=
\|\mathcal{L}^{k/2}f\|_{2}\leq \omega^{k}\|f\|_{2}.
$$
Thus
$$
\mathbf{B}_{\omega}^{2}(\mathcal{L})\subset\mathbf{B}_{\omega}^{2}(\mathbb{D})
=\mathbf{B}_{\omega}(\mathbb{D}).
$$
 On the other hand, if $f$ belongs to $
\mathbf{B}_{\omega/\sqrt{d}}(\mathbb{D})=\mathbf{B}_{\omega/\sqrt{d}}^{2}(\mathbb{D})$
then
$$
\|\mathcal{L}^{k/2}f\|_{2}=\left(\sum_{1\leq i_{1},...,i_{k}\leq
d}\|D_{i_{1}}...D_{i_{k}}f\|_{2}^{2}\right)^{1/2}\leq
\omega^{k}\|f\|_{2},
$$
which together with (\ref{eq00}) gives the embedding $
\mathbf{B}_{\omega/\sqrt{d}}(\mathbb{D})\subset
\textbf{E}_{\omega}(\mathcal{L})=\mathbf{B}_{\omega}^{2}(\mathcal{L}).$
\end{proof}

\begin{thm} If $M$ is equivariantly embedded into $\mathbb{R}^{N}$
then the following equality holds true
$$
\bigcup _{n}\textbf{P}_{n}(M)=\bigcup_{\omega}
\textbf{B}_{\omega}(\mathbb{D}).
$$

\end{thm}

\begin{proof}
 Note that since $\mathcal{L}$ commutes with all
 operators of the form $D_{X}$ where $X$ is 
 a vector field on $M$ defined in (\ref{vf})(see the explanation before the formula (\ref{Laplacian})
in the Introduction) it commutes with the action of $G$ in the
space $L_{2}(M)$. Indeed if $g$ is an element of $G$ then the
action
$$
x\rightarrow g\cdot x, x\in M,
$$
is the same as a translation along integral curve $\exp tX, t\in
\mathbb{R},$ for an appropriate invariant vector field on $M$
(\cite{Z}, Ch. XV, Theorem 8). The corresponding action of $G$ in
the space $L_{2}(M)$ is given by the formula (see the
Introduction)
$$
e^{tD_{X}}f(x)= f(\exp tX\cdot x), t\in \mathbb{R}, x\in M, f\in
C^{\infty}(M).
$$
Thus, we have
$$
\mathcal{L}e^{tD_{X}}f=\mathcal{L}\sum
\frac{(tD_{X})^{k}f}{k!}=\sum
\frac{(tD_{X})^{k}\mathcal{L}f}{k!}=e^{tD_{X}}\mathcal{L}f.
$$
It shows that if $\varphi$ is an eigenfunction with eigenvalue
$\lambda$ then the same is true for $e^{tD_{X}}\varphi$ because
\begin{equation}
\mathcal{L}\left(e^{tD_{X}}\varphi\right)=
e^{tD_{X}}\mathcal{L}\varphi=\lambda
\left(e^{tD_{X}}\varphi\right).\label{inv}
\end{equation}
It implies that all eigen spaces
$\textbf{E}_{\omega}(\mathcal{L})$ are invariant under action of
$G$ in the space $L_{2}(M)$.

We are going to show that if $f\in
\mathbf{B}_{\omega}^{p}(\mathbb{D})$ for a $1\leq p\leq \infty,
\omega>0,$ then $f$ is a polynomial on $M$. Since
 $\mathbf{B}_{\omega}^{p}(\mathbb{D})=\mathbf{B}_{\omega}^{2}(\mathbb{D}),
 1\leq p\leq \infty, \omega>0,$
 we obtain the inequality
 $$
 \|\mathcal{L}^{k}f\|_{2}=\sum_{1\leq i_{1},...,i_{k}\leq d}
 \|D_{i_{1}}^{2}...D_{i_{k}}^{2}f\|_{2}\leq
 (\omega^{2}d)^{k}\|f\|_{2}, m=dim M,
 $$
which shows that $f$ belongs to the space
$\mathbf{B}_{\omega\sqrt{d}}^{2}(\mathcal{L})$. Since by
(\ref{eq00})
$\mathbf{B}_{\omega\sqrt{d}}^{2}(\mathcal{L})=\textbf{E}_{\omega\sqrt{d}}(\mathcal{L})$
we obtain that $f$ belongs to
$\textbf{E}_{\omega\sqrt{d}}(\mathcal{L}).$

 The space
 $\textbf{E}_{\omega\sqrt{d}}(\mathcal{L})$ is finite
 dimensional and according to (\ref{inv}) is invariant under
 the action of $G$ in $L_{2}(M)$.
This fact implies that all translates of every  function $f\in
\mathbf{B}_{\omega}^{p}(\mathbb{D}),1\leq p\leq \infty, \omega>0,$
belong to a finite dimensional space
$\textbf{E}_{\omega\sqrt{d}}(\mathcal{L})$.  In the terminology of
\cite{H1}, \cite{ H2} it means that every   function $f\in
\mathbf{B}_{\omega}^{p}(\mathbb{D}),1\leq p\leq \infty, \omega>0,$
is a $G$-finite vector of the quasi-regular representation of $G$
in $L_{2}(M)$ and by  a result of  S. Helgason \cite{H1}, $\S 3$,
such functions are  restrictions of  polynomials. In other words
we proved the embedding
$$
\bigcup_{\omega>0}\textbf{B}_{\omega}(\mathbb{D})\subset\bigcup_{n\in
\mathbb{N}}\textbf{P}_{n}(M).
$$
Since the Theorem  2.3 implies the opposite embedding we obtain
the desired result. The Theorem is proved.
\end{proof}

\section{Acknowledgment}

I would like to thank the anonymous referee for constructive
suggestions.

 \makeatletter
\renewcommand{\@biblabel}[1]{\hfill#1.}\makeatother

\end{document}